\documentclass[11pt,a4paper]{article}
\usepackage{bbm,times,url,graphicx,anysize,amssymb,amsmath,amsthm,setspace}
\title{Complete subgraphs in multipartite graphs}
\author{\textsc{Florian Pfender}
\\
\normalsize Universit\"at Rostock, Institut f\"ur Mathematik\\
\normalsize D-18051 Rostock, Germany\\
\normalsize \texttt{Florian.Pfender@uni-rostock.de}
\date{}}

\newcommand{\cP}{{\mathcal P}}

\newtheorem{theorem}{Theorem}
\newtheorem{lemma}[theorem]{Lemma}

\newtheorem{corollary}[theorem]{Corollary}

\begin{document}

\maketitle
\begin{abstract} 
Tur\'an's Theorem states that every graph of a certain edge density contains a complete
graph $K^k$ and describes the unique extremal graphs. We give a similar Theorem for
$\ell$-partite graphs. For large $\ell$, we find the minimal edge density $d^k_\ell$,
such that every $\ell$-partite graph whose parts have pairwise edge density greater than
$d^k_\ell$ contains a $K^k$. It turns out that
$d^k_\ell=\frac{k-2}{k-1}$ for large enough $\ell$. We also describe the 
structure of the extremal graphs.
\end{abstract}

\section{Introduction}
All graphs in this note are simple, and we follow the notation of~\cite{D}.
Let $G$ be an $\ell$-partite graph on finite sets $V_1,V_2,\ldots V_\ell$. 
For a vertex $x\in V(G)$, let $d(x):=|N(x)|$. 
The density between two parts is defined as 
$$
d_{ij}:=d(V_i,V_j):=\frac{\| G[V_i\cup V_j]\| }{|V_i||V_j|}.
$$
For a graph $H$ with $|H|\ge \ell$, let $d_\ell(H)$ be the minimum number such that every 
$\ell$-partite graph with $\min_{i<j}d_{ij}>d_\ell(H)$ 
contains a copy of $H$. Clearly, $d_\ell(H)$ is monotone decreasing in $\ell$.
In~\cite{BSTT}, Bondy at al. study the quantity $d_\ell(H)$, and in particular 
$d^3_\ell:=d_\ell(K^3)$, i.e. the values for the complete graph on three vertices, the triangle.
Their main results about triangles can be written as follows.
\begin{theorem} \cite{BSTT}\label{bondy0}
\begin{enumerate}
\item $d^3_3=\tau\approx 0.618$, the golden ration, and
\item $d^3_\omega$ exists and $d^3_\omega=\frac{1}{2}$.
\end{enumerate}
\end{theorem}
They go on and show that $d^3_4\ge 0.51$ and speculate that  $d^3_\ell>\frac{1}{2}$
for all finite $\ell$. We will show that this is false.
In fact, $d^3_\ell=\frac{1}{2}$ for $\ell\ge 13$ as we will proof in Section~\ref{triangles}.
In Section~\ref{complete}, we will extend the proof ideas to show that
$d^k_\ell:=d_\ell(K^k)=\frac{k-2}{k-1}$ for large enough $\ell$.

In order to state our results, we need to define classes
$\mathcal{G}^k_\ell$ of extremal graphs. 
We will do this properly in Section~\ref{extreme}. Our main result is the following theorem.


\begin{theorem}\label{main2}
Let $\ell$ be large enough and let $G=(V_1\cup V_2\cup \ldots \cup V_\ell,E)$
be an $\ell$-partite graph, such that the pairwise edge densities
$$d(V_i,V_j):=\frac{\| G[V_i\cup V_j]\| }{|V_i||V_j|}\ge \frac{k-2}{k-1}\mbox{ for } i\ne j.$$
Then $G$ contains a $K^k$ or $G$ is isomorphic to a graph in $\mathcal{G}^k_\ell$.
\end{theorem}

\begin{corollary}
For $\ell$ large enough, $d^k_{\ell}=\frac{k-2}{k-1}$.
\end{corollary}

The bound on $\ell$ one may get out of the proof is fairly large, and we
think 
that the true bound is much smaller.
For triangles ($k=3$), we can give a reasonable bound on $\ell$. We
think that this bound is not sharp, either. 
In fact we would not be surprised if $\ell\ge 5$ turns out to be sufficient.

\begin{theorem}\label{main}
Let $\ell\ge 13$ and let $G=(V_1\cup V_2\cup \ldots \cup V_\ell,E)$
be an $\ell$-partite graph, such that the pairwise edge densities
$$d(V_i,V_j):=\frac{\| G[V_i\cup V_j]\|}{|V_i||V_j|}\ge \frac{1}{2}\mbox{ for } i\ne j.$$
Then $G$ contains a triangle or $G$ is isomorphic to a graph in $\mathcal{G}^3_\ell$.
\end{theorem}

\begin{corollary}
$d^3_{13}=\frac{1}{2}$.
\end{corollary}

\section{Extremal graphs}\label{extreme}
For $\ell\ge (k-1)!$, a graph $G$ is in $\bar{\mathcal{G}}^k_\ell$, if it
can be constructed as follows. 
\begin{eqnarray*}
V(G)&=&\{ (i,s,t): 1\le i\le \ell, ~1\le s\le k-1, ~1\le t\le n_i^s\} ,\\
E(G)&=&\{ (i,s,t)(i',s',t'): i\ne i',~s\ne s'\} ,
\end{eqnarray*}
where for $\{ \pi_1,\pi_2,\ldots,\pi_{(k-1)!}\}$ being the set of all
permutations of the set $\{ 1,\ldots,k-1\}$, 
\begin{eqnarray*} 
 n_i^{\pi_i(1)}\ge n_i^{\pi_i(2)}\ge\ldots\ge n_i^{\pi_i(k-1)}
&&\mbox{ for } 1\le i\le (k-1)!,  \\
n_i^1 =n_i^2=\ldots=n_i^{k-1} &&\mbox{ for } (k-1)!< i\le
\ell,\mbox{ and} \\
\sum n_i^s>0 &&\mbox{ for } 1\le i\le\ell.
\end{eqnarray*} 
Let $\mathcal{G}^k_\ell$ be the class of graphs which can be obtained from graphs in 
$\bar{\mathcal{G}}_\ell$ by deletion of some edges
in $\{ (i,s,k)(i',s',k'): s\ne s'\land 1\le i<i'\le (k-1)!\}$.

It is easy to see that all graphs in $\mathcal{G}^k_\ell$
are $\ell$-partite and that $\mathcal{G}^k_\ell$ contains graphs with $d_{ij}\ge \frac{k-2}{k-1}$ 
(and  $d_{ij}= \frac{k-2}{k-1}$ for $j>(k-1)!$ for all graphs in $\mathcal{G}^k_\ell$). 

For $k=3$, the density condition is fulfilled for all graphs in 
$\bar{\mathcal{G}}^3_\ell$, and for all graphs in  $\mathcal{G}^3_\ell$ which have $d_{12}\ge \frac{1}{2}$.

For $k>3$, this description is not a full characterization of the extremal graphs in the problem. We would need some extra conditions on the $n_i^s$ to make sure that all graphs in  $\bar{\mathcal{G}}^k_\ell$ fulfill the density conditions.

\section{Theorem~\ref{main}---triangles}\label{triangles}
In this section we prove Theorem~\ref{main}.
We will start with a few useful lemmas.
An important lemma for the study of $d^3_\omega$ is the following.
\begin{lemma} \cite{BSTT}\label{bondy'}
Let $G=(V_1\cup V_2\cup V_3\cup V_4,E)$ be a $4$-partite
graph with $|V_1|=1$, such that the pairwise edge densities
$d(V_i,V_j)>\frac{1}{2}\mbox{ for } i\ne j.$
Then $G$ contains a triangle.
\end{lemma}
With the same proof one gets a slightly stronger result which we will use in our proof. 
In most cases occurring later, $X$ will be the neighborhood of a vertex, and the Lemma will be used
to bound the degree of the vertex.
\begin{lemma}\label{bondy}
Let $G=(V_1\cup V_2\cup V_3,E)$ be a $3$-partite
graph and $X$ an independent set, such that the pairwise edge densities
$d(V_i,V_j)\ge \frac{1}{2}$ for $ i\ne j$ and $|X\cap V_i|\ge
\frac{1}{2}|V_i|$ for $1\le i\le 3$, with a strict
inequality for at least two of the six inequalities.
Then $G$ contains a triangle.
\end{lemma}
In order to prove the first part of Theorem~\ref{bondy0}, the authors show a stronger result.
\begin{theorem} \cite{BSTT}\label{bondy2}
Let $G=(V_1\cup V_2\cup V_3,E)$ be a $3$-partite
graph with edge densities
$d_{ij}:=d(V_i,V_j)$, and $d_{ij}d_{ik}+d_{jk}>1$ for $\{i,j,k\}=\{1,2,3\}$.
Then $G$ contains a triangle.
\end{theorem}
As a corollary from Lemma~\ref{bondy} and Theorem~\ref{bondy2} we get
\begin{corollary}\label{bondy3}
Let $G=(V_1\cup V_2\cup\ldots \cup V_\ell,E)$ be a balanced $\ell$-partite graph
on $n\ell$ vertices with edge densities $d_{ij}\ge \frac{1}{2}$, which does not contain a triangle.
Then for every independent set $X\subseteq V(G)$, $|X|\le \frac{(\ell+1)n}{2}$.
\end{corollary}




\begin{proof}[Proof of Theorem~\ref{main}]
Suppose that $G$ contains no
triangle. Without loss of generality we may assume that each of the
$\ell\ge 13$ parts of $G$ contains exactly $n$ vertices, where $n$ is
a sufficiently large even integer. Otherwise, blow up each part by
an appropriate factor, which has no effect on the densities or the
membership in $\mathcal{G}_\ell$, and creates no triangles. 

For a vertex $x$ let $d_i(x)=|N(x)\cap V_i|$.
For each edge $xy\in E(G)$, choose $i$ and $j$ such that
$x\in V_i$ and $y\in V_j$, and let
$$
s(xy):=d(x)-d_j(x)+d(y)-d_i(y).
$$
We have
$$
\sum_{xy\in E(G)}s(xy)= 
\frac{1}{2}\sum_{\substack{x\in V(G)\\y\in N(x)}}s(xy) 
=\sum_{x\in V(G)}\left( d(x)^2-\sum_{j=1}^\ell d_j(x)^2\right) .
$$
The set $N(x)$ is independent, so by Lemma~\ref{bondy}, at most two of
the $d_j(x)$ may be larger than 
$\frac{n}{2}$, and by Lemma~\ref{bondy2}, $d_j(x)d_k(x)\le \frac{1}{2}n^2$ for
every vertex $x\in V_i$ and $j\ne k$. 

Therefore, for fixed $d(x)\ge n$, the
last sum is minimized if $d_j(x)=n$ for one $j$,
$d_j(x)=\frac{n}{2}$ or $d_j(x)=0$ for all but one of the other $j$,
and $0\le d_j(x)\le \frac{n}{2}$ for the last remaining $j$. For $d(x)< n$, the
last sum is non negative. Thus,
\begin{eqnarray*}
\frac{1}{\| G\|}\sum_{xy\in E(G)}s(xy)&\ge& 
\frac{2}{\sum d(x)}
\sum_{\substack{x\in V(G)
}}
(d(x)^2-n^2-(d(x)-n)\tfrac{n}{2})\\
&=& \frac{2\sum d(x)^2}{\sum d(x)}-n-\frac{\ell n^3}{\sum d(x)}\\
&\ge& \frac{2}{\ell n}\sum d(x)-n-\frac{\ell n^3}{\sum d(x)}\\
&\ge& (\ell-2)n -\frac{2n}{\ell-1}.
\end{eqnarray*}
Therefore, there is an edge $xy\in E(G)$ with $s(xy)\ge  (\ell-2)n -\frac{2n}{\ell-1}$. By symmetry, 
we may assume that $x\in V_1$, $y\in V_2$ and $d(x)-d_2(x)\ge d(y)-d_1(y)$. 
Let 
$$
N'(x):= N(x)\setminus V_2,~N'(y):= N(y)\setminus V_1,~\mbox{and}~
W':= \bigcup_{i=3}^\ell V_i\setminus(N(x)\cup N(y)).
$$
Let $G':=G[\bigcup_{i=3}^\ell V_i]$.
Since $N'(x)$ and $N'(y)$ are independent sets, and $|W'|\le \frac{2n}{\ell-1}\le \frac{n}{6}$,
and by Lemma~\ref{bondy} and Theorem~\ref{bondy2}, for fixed $|W'|$
$G'[N'(x)\cup N'(y)]$ has at most as many edges as in the graph we would get if 
$|N(x)\cap V_3|= |N(y)\cap V_4|= n$ and 
$|(W'\cup N'(Y))\cap V_5|=|N(x)\cap V_5|=|N(x)\cap V_i|=|N(y)\cap V_i|=\frac{n}{2}$ for $6\le i\le \ell$,
and all possible edges (i.e., all edges not inside $N(x)$, $N(y)$ or one of the $V_i$) are there.
So,
$$
\| G'[N'(x)\cup N'(y)]\| \le {\ell-2\choose 2}\frac{n^2}{2}+\frac{n^2}{2}-|W'|\frac{\ell-3}{2}n.
$$
Further, by Corollary~\ref{bondy3}, no vertex in $G'$ can have degree larger than $\frac{\ell-2}{2}n$,
so
$$
\| G'\| \le {\ell-2\choose 2}\frac{n^2}{2}+\frac{n^2}{2}+|W'|\frac{n}{2}
\le  {\ell-2\choose 2}\frac{n^2}{2}+\frac{7}{12}n^2.
$$
On the other hand, by the density condition,
$$
\| G'\| \ge  {\ell-2\choose 2}\frac{n^2}{2},
$$
so at most  $\frac{7}{12}n^2$ 
of the possible edges between $N'(x)$ and $N'(y)$ are missing. In particular,
no vertex $z$ can have large neighborhoods in both $N'(x)$ and $N'(y)$, i.e.
$$
(|N(z)\cap N'(x)|-n)|N(z)\cap N'(y)|<\frac{7}{12}n^2 
.
$$
Let
\begin{eqnarray*}
X'&:=&\{ v\in V(G'):|N(v)\cap N'(x)|> \tfrac{1}{2}|N'(x)|\},\\
Y'&:=&\{ v\in V(G'):|N(v)\cap N'(y)|> \tfrac{1}{2}|N'(y)|\},~\mbox{and}\\
Z'&:=&V(G')\setminus(X'\cup Y').
\end{eqnarray*}
If $z\in Z'$, then
$$
d_{G'}(z)\le \frac{|N'(x)|}{2}+\frac{7n^2}{6|N'(x)|-12n}+|W'| 
\le \frac{(\ell-1)n}{4}+\frac{7n}{3\ell-15}+\frac{2n}{\ell-1},
$$
and  for $z\in Z'\setminus W'$,
at least $\frac{|N'(y)|}{2}-n\ge  \frac{\ell-7}{4}n\ge \frac{3}{2}n$ of the missing possible
edges edges between $N'(x)$ and $N'(y)$ are incident to $z$. Therefore,
$|Z'|\le |W'|+      \frac{7}{9}n\le \frac{17}{18}n$.



Again, since $X'$ is an independent set, at most two of the sets $V_i\cap X'$ contain more than 
$\frac{n}{2}$ vertices. We may assume that these sets are contained in $V_3\cup V_4$.
Let $G''=G'\setminus (V_3\cup V_4)$, and $X''$, $Y''$ and $Z''$ 
the according subsets of 
$X'$, $Y'$ and $Z'$.
Then by the density condition,
$$
\| G''\| \ge  {\ell-4\choose 2}\frac{n^2}{2}.
$$
On the other hand,
$
\| G''\| \le |E(X'',Y'')|+|E(Z'',V(G'')|,
$
and $|E(X'',Y'')|$ is maximized for fixed $|Z''|<n$
if $|V_5\cap Y''|=n$ and $|V_i\cap X''|=\frac{n}{2}$ for $6\le i\le \ell$.
Thus,
\begin{eqnarray*}
\| G''\| &\le& \frac{n^2(\ell-5)}{2}+\frac{n(\ell-6)}{2}\left( \frac{n(\ell-5)}{2}-|Z''|\right) 
+|Z''|\left(  \frac{\ell-1}{4}+\frac{7}{3\ell-15}+\frac{2}{\ell-1}\right) n\\
&=&{\ell-4\choose 2}\frac{n^2}{2}+|Z''|\left(  \frac{\ell-1}{4}+\frac{7}{3\ell-15}+\frac{2}{\ell-1}-\frac{\ell-6}{2}\right) n \\
&\le&  {\ell-4\choose 2}\frac{n^2}{2}.
\end{eqnarray*}
Equality is only attained for $Z''=\emptyset$, in which case it is easy to show 
that $G$ is isomorphic to a graph in $ \mathcal{G}^3_\ell$.
\end{proof}




\section{Theorem~\ref{main2}---complete subgraphs}\label{complete}



Graphs which have almost enough edges to force a $K^k$ either contain
a $K^k$ or have a structure very similar to the Tur\'an graph.
This is described by the following theorem
from~\cite{BB}, where a more general version is credited to Erd\"os and
Simonovits. 

\begin{theorem}\label{erd}\cite[Theorem VI.4.2]{BB}
Let $k\ge 3$. Suppose a graph $G$ contains no $K^k$ and
$$
\| G\| =\left( 1-\frac{1}{k-1}+o(1)\right){|G|\choose 2} .
$$
Then $G$ contains a $(k-1)$-partite graph of minimal degree
$(1-\frac{1}{k-1}+o(1))|G|$ as  
an induced subgraph.
\end{theorem}

%

\begin{proof}[Proof of Theorem~\ref{main2}]
For the ease of reading and since we are not trying to
minimize the needed $\ell$, we will use some variables $\ell_i$ and
$c_i>0$. As $\ell$ is chosen larger, 
the $\ell_i$ grow without bound and the $c_i$ approach $0$.

Let $G$ be an $\ell$-partite graph with $V(G)=V_1\cup V_2\cup\ldots \cup V_\ell$ 
with densities $d_{ij}\ge \frac{k-2}{k-1}$, and suppose that $G$ contains no $K^k$.
Without loss of generality we may assume that each of the $V_i$
contains exactly $n$ vertices, where $n$ is
a sufficiently large integer divisible by $k-1$. 

We have
$$
\| G\| \ge \left( 1-\frac{1}{k-1}-\frac{1}{\ell}\right){|G|\choose 2}.
$$
Let $H$ be the $(k-1)$-partite subgraph of $G$ guaranteed by Theorem~\ref{erd}, with 
$V(H)=X_1\cup X_2\cup\ldots\cup X_{k-1}$ and $Z:=V(G)\setminus V(H)$. There is a $c_1>0$
so that $|Z|\le c_1|G|$ (and this $c_1$ becomes arbitrarily small if $\ell$ is chosen large enough). Let
$X_{i,j}:=V_i\cap X_j$ and $Z_i:=V_i\setminus\bigcup_j X_{i,j}$. After renumbering the $V_i$
and the $X_j$, 
we have $|Z_i|\le 2c_1n$ and $|X_{i,1}|\ge |X_{i,2}|\ge\ldots\ge |X_{i,k-1}|$  for
$1\le i\le \ell_1\le \ell$, where $\ell_1\ge\frac{\ell}{2(k-1)!}$ is picked as large as possible. For some $c_2>0$ (with $c_2\to 0$), there is at most 
one index $i\le \ell_1$ with $|X_{i,1}|> \left( \frac{1}{k-1}+c_2\right) n$,
as otherwise there is a pair $(V_i,V_{i'})$ with $d_{ii'}<\frac{k-2}{k-1}$.
So we may assume that
$$
\left( \frac{1}{k-1}-kc_2\right) n\le |X_{i,j}|\le \left( \frac{1}{k-1}+c_2\right) n
$$
for $1\le i\le \ell_1-1$ and $1\le j\le k-1$. This implies that
$$
\|G[X_{i,j},X_{i',j'}]\|>|X_{i,j}||X_{i',j'}|-c_3n^2 
$$
for $i\ne i'$, $j\ne j'$, $1\le i,i'\le \ell_1-1$, $1\le j,j'\le k-1$ and some $c_3>0$ with $c_3\to 0$.

For every $v\in \bigcup_{i\le \ell_1-1}V_i$, find a maximum set $\cP_v$ of pairs $( i_s,j_s)$ with
$(1,1)\le ( i_s,j_s)\le (\ell-1,k-1)$, $i_s\ne i_{s'}$, $j_s\ne j_{s'}$, $|N(v)\cap X_{i_s,j_s})|>c_4n$, where $c_4:= k\sqrt{c_3}$. If there is a vertex $v$ with  $|\cP |=k-1$, then we have a $K^k$. So we may assume this is not the case. Assign $v\in Z$ to one
set $Y_j\supseteq X_j\cap \bigcup_{i\le \ell_1-1}V_i$, if there is no pair $(i,j)$ in $\cP_v$. 
If there is more than one available set, arbitrarily pick one.

Now we reorder the $V_i$ and $Y_j$ again to guarantee that
$|Y_{i,1}|\ge \ldots \ge |Y_{i,k-1}|$ for $1\le i\le \ell_2< \ell_1$, with $\ell_2\ge\frac{\ell_1-1}{(k-1)!}$ as large as possible. 
In the following, only consider indices $i\le \ell_2$.
Note that for $v\in Y_{i,j}$, $|N(v)\cap Y_{i,j'}|<(c_4+2c_1)n$ for all but at most $k-2$ different $j'$, 
as $Y_{i,j'}\setminus X_{i,j'}\subseteq Z_{j'}$.

Let $Y_i'\subseteq Y_i$ the set of all vertices $v\in Y_i$ with $|N(v)\cap Y_j)|<\frac{1}{2}(\frac{1}{k-1}+c_5)\ell_2n$ for some $j\ne i$, $c_5:=c_2+c_4$. 
Note that the sets $Y_i\setminus Y_i'$ are independent, as the intersection of the neighborhoods of every two vertices in this set contain a $K^{k-2}$. Every vertex in $v\in Y_i'\cap V_j$ may have up to $((c_4+2c_1)(\ell_2-k+1)+k-2)n$ neighbors in $Y_i$. But, at the same time, $v$ has at least $|Y_{i'}|-\frac{1}{2}(\frac{1}{k-1}+c_5)\ell_2n-n>\frac{1}{3k}\ell_2n$ non-neighbors in some $Y_{i'}\setminus V_j$, $i'\ne i$. Then
\begin{eqnarray*}
\| G[V_1\cup\ldots V_{\ell_2}]\|
&\le & \sum_{\substack{i\ne i'\\ j<j'}}|Y_{i,j}||Y_{i',j'}| 
+\sum_i |Y_i'|\left( ((c_4+2c_1)(\ell_2-k+1)+k-2)n-\tfrac{1}{3k}\ell_2n\right)\\
&\le& \sum_{\substack{i\ne i'\\ j<j'}}|Y_{i,j}||Y_{i',j'}| 
+\sum_i
|Y_i'|\underbrace{(c_4+2c_1+\tfrac{k}{\ell_2}-\tfrac{1}{3k})}_{<0\mbox{
  for large enough }\ell}\ell_2n\\
&\le & {\ell_2\choose 2}n^2-\sum_{j<j'}|Y_{i,j}||Y_{i,j'}| \\
&\le & {\ell_2\choose 2}\tfrac{k-2}{k-1}~n^2,
\end{eqnarray*}
where equality only holds if $|Y_{i}'|=0$ for all $i$, and $|Y_{i,j}|=\frac{n}{k-1}$ for $1\le j\le k-1$ and all
but at most one index $i$.

This completes the proof of $d^k_\ell=\frac{k-2}{k-1}$ for large enough $\ell$. We are left to analyze the extremal graphs. After reordering, we have $|Y_{i,j}|=\frac{n}{k-1}$ and $d(Y_{i,j},Y_{i',j'})=1$ for $1\le j,j'\le k-1$ and $1\le i,i'\le k$, if $i\ne i'$ and $j\ne j'$. 

Let $v\in V_{i'}$ for some $i'>k$. Then $|N(v)\cap\bigcup_{i\le k}V_i|\le \frac{k(k-2)}{k-1}n$, as otherwise there is a $K^{k-1}$ in $N(v)$. On the other hand, equality must hold for all vertices $v\in V_{i'}$ due to the density condition. Therefore, $N(v)\cap\bigcup_{i\le k}V_i=V_i\setminus Y_j$ for some $1\le j\le k-1$. Define $Y_{i',j}$ accordingly for all $i'>k$, and let $Y_j=\bigcup_i Y_{i,j}$. Then $V=\bigcup  Y_{j}$. For every permutation $\pi$ of the set $\{ 1,\ldots , k-1\}$, there can be at most one set $V_i$ with $|Y_{i,\pi(1)}|\ge |Y_{i,\pi(2)}|\ge\ldots\ge |Y_{i,\pi(k-1)}|$ and $|Y_{i,\pi(1)}|>|Y_{i,\pi(k-1)}|$. Otherwise, this pair of sets would have density smaller than $\frac{k-2}{k-1}$. Thus, all but at most $(k-1)!$ of the $V_i$ have  $|Y_{i,j}|=\frac{n}{k-1}$ for $1\le j\le k-1$. Therefore, all extremal graphs are in $\mathcal{G}^k_\ell$.
\end{proof}

\section{Open problems}

As mentioned above, the characterization of the extremal graphs is not
complete for $k>3$. We need to determine all parameters $n_i^s$ so
that the resulting graphs in $\bar{\mathcal{G}}^k_\ell$ fulfill the
density conditions.

The other obvious question left open is a good bound on $\ell$
depending on $k$ in Theorem~\ref{main2}, and the determination of the exact values of
$d_\ell^k$ for smaller $\ell$. In particular, is it true that
$d_5^3=\frac{1}{2}$?

Another interesting open topic is the behavior of $d_\ell(H)$ for
non-complete $H$. Bondy et al.~\cite{BSTT} show that
$$\lim_{\ell\to\infty}d_\ell(H)=\frac{\chi(H)-2}{\chi(H)-1},$$
but it should be possible to show with similar methods as in this
note that $d_\ell(H)=\frac{\chi(H)-2}{\chi(H)-1}$ for large enough
$\ell$ depending on $H$.

\bibliographystyle{amsplain}

\end{document}